\documentclass[12pt]{amsart}

\usepackage{tikz}

\usepackage[margin=1in,letterpaper,portrait]{geometry}

\theoremstyle{definition}
\newtheorem{thm}{Theorem}[section]
\newtheorem{prop}[thm]{Proposition}

\newtheorem{conj}[thm]{Conjecture}
\newtheorem{lemma}[thm]{Lemma}

\newtheorem{rem}[thm]{Remark}

\DeclareMathOperator{\des}{des}
\DeclareMathOperator{\asc}{asc}
\DeclareMathOperator{\inv}{inv}

\DeclareMathOperator{\val}{val}
\DeclareMathOperator{\sgn}{sgn}
\DeclareMathOperator{\id}{id}

\newcommand{\el}[2]{\genfrac{\langle}{\rangle}{0pt}{}{#1}{#2}}

\newcommand{\RR}{\mathbb{R}}

\newcommand{\pmE}{$\pm$-Eulerian}

\title{On the joint distribution of descents and signs of permutations}

\author[Fulman]{Jason Fulman}%
\address{Department of Mathematics, USC, Los Angeles, CA}
\email{fulman@math.usc.edu}

\author[Kim]{Gene B. Kim}%
\address{Department of Mathematics, Stanford University, Stanford, CA}
\email{genebkim@stanford.edu}

\author[Lee]{Sangchul Lee}%
\address{Department of Mathematics, UCLA, Los Angeles, CA}
\email{sangchul87.lee@gmail.com}

\author[Petersen]{T. Kyle Petersen}%
\address{Department of Mathematical Sciences, DePaul University, Chicago, IL}
\email{tpeter21@depaul.edu}

\date{January 30, 2021}

\begin{document}

\maketitle

\begin{abstract}
We study the joint distribution of descents and sign for elements of the symmetric group and the hyperoctahedral group (Coxeter groups of types $A$ and $B$). For both groups, this has an application to riffle shuffling: for large decks of cards the sign is close to random after a single shuffle. In both groups, we derive generating functions for the Eulerian distribution refined according to sign, and use them to give two proofs of central limit theorems for positive and negative Eulerian numbers.
\end{abstract}

\section{Introduction}

The distribution of descents over all permutations is known as the \emph{Eulerian distribution}, and the number of permutations of $n$ elements with a given number of descents is known as an \emph{Eulerian number}. The Eulerian numbers are ubiquitous in combinatorics; see \cite{Pet} for an entire book devoted to Eulerian numbers, as well as various refinements and generalizations of them.

In this paper we study one such refinement, namely the joint distribution of descents and sign, which has been studied, e.g., by Tanimoto \cite{Tan}, and more recently by Dey and Sivasubramanian \cite{DS}. Our work is also in some sense an extension of work of Loday \cite{Lo} and Desarm\'enien and Foata \cite{DF}, who studied ``signed'' Eulerian numbers. This distribution is also briefly mentioned as an example in Section 9.3 of the sweeping work by Hwang, Chern, and Duh \cite{HCD} who give a unified approach to central limit theorems for distributions satisfying Eulerian-type recurrences.

Permutations of the set $[n] = \{1,2,\ldots,n\}$ provide one combinatorial description for the elements of the symmetric group $S_n$. In Section \ref{sec:Bn} we also consider the analogous joint distributions of descents and sign over elements of the hyperoctahedral group $B_n$ (and more generally, one can replace sign by a one-dimensional character of $B_n$). The symmetric and hyperoctahedral groups are Coxeter groups of types $A_{n-1}$ and $B_n$, and the distribution of descents is well-understood in any finite Coxeter group; see, e.g., the recent paper of Kahle and Stump \cite{KS} for the current state of knowledge. However, to our knowledge the limiting distribution of descents and sign has not been investigated outside of the symmetric group, and this paper is the first to make connections between such distributions and card shuffling.

\subsection{Basic definitions}

We now provide some definitions in the symmetric group case. Definitions for the hyperoctahedral group are given in Section \ref{sec:Bn}. First, we define a permutation $w$ to be a bijection $[n]\to [n]$, which we write in one-line notation: $w=w(1)w(2)\cdots w(n)$. We let $S_n$ denote the set of all permutations of $[n]$. A \emph{descent} of a permutation is a position $i$ such that $w(i)>w(i+1)$. We let $\des(w)$ denote the number of descents of $w$, i.e., for $w \in S_n$,
\[
 \des(w) = |\{ 1\leq i \leq n-1 : w(i) > w(i+1) \}|.
\]
Similarly, an \emph{ascent} of a permutation is a position $i$ such that $w(i)<w(i+1)$, and the number of ascents is
\[
 \asc(w) = |\{ 1\leq i \leq n-1 : w(i) < w(i+1) \}|.
\]
If a permutation has $k$ descents, then it has $n-1-k$ ascents, while its reversal, $\overleftarrow{w}=w(n)w(n-1)\cdots w(1)$ has $k$ ascents. Thus, the permutations with $k$ descents are in bijection with the permutations having $k$ ascents. Note also that if $w$ has $k$ descents, the descent positions partition $w$ into $k+1$ maximally increasing runs. For example, the permutation $w=4|3|126|5$ has $\des(w)=3$, and the descent positions (indicated with vertical bars) partition $w$ into $4$ increasing runs.

The Eulerian numbers are denoted $\el{n}{k}$, $1\leq k \leq n$, and they count the number of permutations with $k$ increasing runs, i.e., with $k-1$ descents. That is,
\[
 \el{n}{k} = |\{ w \in S_n : \des(w) = k-1\}|.
\]
The first few rows of the Eulerian triangle are shown in Table \ref{tab:eul}.

The \emph{Eulerian polynomials} are the generating functions for the rows of this triangle, i.e., the generating function for permutations according to the descent statistic:
\[
 A_n(t) = \sum_{k=1}^n \el{n}{k} t^k =  \sum_{ w \in S_n} t^{\des(w)+1}.
\]
For example, we have $A_3(t) = t+4t^2 + t^3$ and $A_4(t) = t+11t^2 + 11t^3 + t^4$.

\begin{table}[t]
\begin{center}
\begin{tabular}{c |rrrrrrrr}
$n\backslash k$ & 1 & 2 & 3 & 4 &  5 & 6 & 7\\
\hline
1 & 1  \\
2 & 1 & 1 \\
3 & 1 & 4 & 1 \\
4 & 1 & 11 & 11 & 1 \\
5 & 1 & 26 & 66 & 26 & 1 \\
6 & 1 & 57 & 302 & 302 & 57 & 1  \\
7 & 1 & 120 & 1191 & 2416 & 1191 & 120 & 1\\
\end{tabular}
\end{center}
\caption{Triangle of the Eulerian numbers $\el{n}{k}$, the number of permutations in $S_n$ with $k-1$ descents.}\label{tab:eul}
\end{table}

The \emph{sign} of a permutation is $1$ if it can be written as a product of an even number of transpositions; otherwise the sign is $-1$. The sign of a permutation is well-defined, and denoted by $\sgn(w)$. It is also related to other statistics for permutations, e.g.,
\[
\sgn(w) = (-1)^{n-c(w)} = (-1)^{\inv(w)},
\]
where $c(w)$ is the number of cycles of $w$, and $\inv(w)$ is the number of inversions of $w$. Let $S_n^+$ denote the set of permutations of positive sign (also known as ``even'' permutations), and let $S_n^- = S_n - S_n^+$ denote the set of permutations with negative sign, i.e.,
\[
 S_n^+ = \{ w \in S_n : \sgn(w) = 1\} \quad \mbox{and} \quad S_n^- = \{ w \in S_n : \sgn(w) = -1\}.
\]

We now define the \emph{positive Eulerian number}, denoted $\el{n}{k}^+$, to be the number of permutations in $S_n^+$ with $k-1$ descents and define the \emph{negative Eulerian number}, denoted $\el{n}{k}^-$, to be the number of permutations in $S_n^-$ with $k-1$ descents, i.e.,
\[
 \el{n}{k}^+ = |\{w \in S_n^+ : \des(w) = k-1\}| \quad \mbox{and} \quad \el{n}{k}^- = |\{w \in S_n^- : \des(w) = k-1\}|.
\]
The first few rows of the positive and negative Eulerian numbers are shown in Tables \ref{tab:poseul} and \ref{tab:negeul}.

\begin{table}[t]
\begin{center}
\begin{tabular}{c |rrrrrrrr}
$n\backslash k$ & 1 & 2 & 3 & 4 &  5 & 6 & 7\\
\hline
1 & 1  \\
2 & 1 & 0 \\
3 & 1 & 2 & 0 \\
4 & 1 & 5 & 5 & 1 \\
5 & 1 & 14 & 30 & 14 & 1 \\
6 & 1 & 29 & 147 & 155 & 28 & 0  \\
7 & 1 & 64 & 586 & 1208 & 605 & 56 & 0\\
\end{tabular}
\end{center}
\caption{Triangle of the positive Eulerian numbers, $\el{n}{k}^+$, the number of permutations in $S_n^+$ with $k-1$ descents.}\label{tab:poseul}
\end{table}

\begin{table}[t]
\begin{center}
\begin{tabular}{c |rrrrrrrr}
$n\backslash k$ & 1 & 2 & 3 & 4 &  5 & 6 & 7\\
\hline
1 & 0  \\
2 & 0 & 1 \\
3 & 0 & 2 & 1 \\
4 & 0 & 6 & 6 & 0 \\
5 & 0 & 12 & 36 & 12 & 0 \\
6 & 0 & 28 & 155 & 147 & 29 & 1  \\
7 & 0 & 56 & 605 & 1208 & 586 & 64 & 1\\
\end{tabular}
\end{center}
\caption{Triangle of the negative Eulerian numbers, $\el{n}{k}^-$, the number of permutations in $S_n^-$ with $k-1$ descents.}\label{tab:negeul}
\end{table}

We define the \emph{positive Eulerian polynomial}, $A_n^+(t)$, and the \emph{negative Eulerian polynomial}, $A_n^-(t)$, to be the generating functions for the positive and negative Eulerian numbers, respectively. That is,
\[
 A_n^+(t) = \sum_{k=1}^n \el{n}{k}^+ t^k =  \sum_{ w \in S_n^+} t^{\des(w)+1},
\]
and
\[
A_n^-(t) = \sum_{k=1}^n \el{n}{k}^- t^k =  \sum_{ w \in S_n^-} t^{\des(w)+1}.
\]
For example, $A_3^+(t) = t+2t^2$, $A_3^-(t) = 2t^2+t^3$, $A_4^+(t) = t+5t^2+5t^3+t^4$, and $A_4^-(t) = 6t^2+6t^3$. Throughout the paper, when referring generically to both the positive Eulerian numbers and the negative Eulerian numbers, or to their corresponding polynomial generating functions, we will write ``\pmE{} numbers'' or ``\pmE{} polynomials'' and use the notation $\el{n}{k}^{\pm}$ and $A_n^{\pm}(t)$.

\subsection{Main results for the symmetric groups}

This paper contains three theorems and a conjecture, and each theorem is proved in more than one way. The conjecture, while interesting in its own right, would provide a third proof of one of the theorems. In Section \ref{sec:Bn} we have analogous results for the hyperoctahedral group.

Among the many identities for Eulerian numbers is the following power series identity:
\begin{equation}\label{eq:seriesA}
 \frac{A_n(t)}{(1-t)^{n+1}} = \sum_{k\geq 0} k^n t^k.
\end{equation}
Our first main result for the \pmE{} numbers is a similar identity, given in the following theorem.

\begin{thm}[Generating function identity]\label{thm:main}
 For all $n \geq 1$,
 \begin{equation}\label{eq:seriesApm}
 \frac{A_n^{\pm}(t)}{(1-t)^{n+1}} = \sum_{k\geq 0} \frac{k^n \pm k^{\lceil n/2 \rceil}}{2} t^k.
 \end{equation}
\end{thm}

We will provide two different proofs of this result in Section \ref{sec:identities}. We will also draw two interesting conclusions from Theorem \ref{thm:main}. The first of these is a central limit theorem for the distribution of \pmE{} numbers, first obtained by Hwang, Chern, and Duh \cite{HCD}.

\begin{thm}[Limiting distribution]\label{thm:clt}
The distribution of the coefficients of $A_n^{\pm}(t)$ is asymptotically normal
as $n \rightarrow \infty$. For $n \geq 4$, these numbers have mean $(n+1)/2$ and for $n \geq 6$, these numbers have variance $(n+1)/12$.
\end{thm}

While one proof of Theorem \ref{thm:clt} will use Theorem \ref{thm:main} (essentially, this is the same idea outlined in \cite[Section 9.3]{HCD}), we provide another proof in Section \ref{sec:clt}. In fact we conjecture that the polynomials $A_n^{\pm}(t)$ have all roots real, which by Harper's method \cite{Pit} would give a third proof of our central limit theorem.

\begin{conj}[Real roots]\label{conj:real}
The \pmE{} polynomials have only real roots: $A_n^+(t)$ has real roots for $n\geq 1$ and $A_n^-(t)$ has real roots for $n\geq 2$.
\end{conj}

We have verified Conjecture \ref{conj:real} for $n\leq 100$. We provide some remarks on the conjecture in Section \ref{sec:roots}.

The second conclusion that we can draw from Theorem \ref{thm:main} has to do with card shuffling.

\begin{thm}[Sign after a riffle shuffle]\label{thm:shuffle}
For $a$-shuffling starting at the identity, the probability of having sign $1$ after $k$ steps is equal to
\[
 \frac{1}{2} + \frac{1}{2a^{k \lfloor n/2 \rfloor}} .
\]
\end{thm}

By shuffling here we mean the Gilbert-Shannon-Reeds (GSR) model of riffle shuffling, the details of which will be provided in Section \ref{sec:shuffle}. In that section we also present two further proofs of Theorem \ref{thm:shuffle}, including one that follows immediately from work of Amy Pang. A similar result for shelf-shuffling machines is also discussed there, which says that (for ultimately trivial reasons) the probability of having sign $1$ after one pass through a shelf-shuffling machine is \emph{exactly} $1/2$.

\subsection*{Acknowledgements} Fulman was partially supported by Simons Foundation Grant 400528. Petersen was partially supported by Simons Foundation Collaboration Travel Grant 353772 and by a DePaul University College of Science and Health Faculty Summer Research Grant. The authors thank Persi Diaconis, Amy Pang, and Victor Reiner for helpful discussions.

\section{Identities for \pmE{} numbers and polynomials}\label{sec:identities}

In this section we will prove Theorem \ref{thm:main} and discuss some identities and recurrences for \pmE{} numbers and polynomials.

\subsection{Proofs of Theorem \ref{thm:main}}\label{sec:mainproof}

We will present two proofs of Theorem \ref{thm:main}. The first proof uses a generating function for the joint distribution of descents and cycle structure studied by Fulman \cite{F} and the second uses results of Desarm\'enien and Foata \cite{DF}, and Wachs \cite{W}, for the joint distribution of descents and inversions.

We begin the first proof by recalling an identity from \cite[Theorem 1]{F}:
\begin{equation}\label{eq:F1}
\sum_{n \geq 0} \frac{u^n \sum_{w \in S_n} t^{\des(w)+1} \prod_i x_i^{n_i(w)}}{(1-t)^{n+1}}
= \sum_{k \geq 1} t^k \prod_j (1-x_ju^j)^{-f_{j,k}}.
\end{equation}
In this identity, the $x_i$ are indeterminates, $n_i(w)$ is the number of $i$-cycles in the permutation $w$, and
\[
f_{j,k} = \frac{1}{j} \sum_{d|j} \mu(d) k^{j/d},
\]
where $\mu$ is the M\"obius function of elementary number theory. While we do not make use of the fact here, the quantity $f_{j,k}$ counts the number of primitive necklaces of length $j$ drawn from an alphabet of $k$ letters. We will also make use of the following lemma, which is found Marshall Hall's group theory book \cite{Hall}, in connection with the commutator calculus on the free group.

\begin{lemma}\label{lem:usesim}
For any integer $k\geq 1$, we have the following power series identity:
\[ \prod_{j \geq 1} (1-u^j/k^j)^{-f_{j,k}} = \frac{1}{1-u}. \]
\end{lemma}

The upshot for us comes from setting all $x_i=-1$ and $u=-u$ in Equation \eqref{eq:F1}. Note that $\prod_i (-1)^{n_i(w)} = (-1)^{c(w)}$, where $c(w)$ is the number of cycles in $w$. Since $\sgn(w) = (-1)^{n-c(w)}=(-1)^{n+c(w)}$, we have the following series identity:
\begin{equation}\label{eq:Fspecial}
 \sum_{n \geq 0} \frac{u^n \sum_{w \in S_n} t^{\des(w)+1} \sgn(w)}{(1-t)^{n+1}}
 = \sum_{k \geq 1} t^k \prod_{j \geq 1} (1+(-u)^j)^{-f_{j,k}},
\end{equation}
which is key in what follows.

\begin{proof}[First proof of Theorem \ref{thm:main}]
Since Equation \eqref{eq:seriesA} gives
\[ \sum_{w \in S_n} t^{\des(w)+1} = (1-t)^{n+1} \sum_{k \geq 1} t^k k^n,\]
it is enough to prove that
\begin{equation} \label{eq:needprove}
\sum_{w \in S_n} t^{\des(w)+1} \sgn(w) = (1-t)^{n+1} \sum_{k \geq 1}
 t^k k^{\lceil n/2 \rceil}.
\end{equation}

We now attack the product on the right-hand side of Equation \eqref{eq:Fspecial} with algebra and an application of Lemma \ref{lem:usesim}. Letting $v=ku$, we have:
\begin{align*}
\prod_{j \geq 1} \left(\frac{1}{1+u^j}\right)^{f_{j,k}} &= \prod_{j \geq 1} \left( \frac{1-u^j}{1-u^{2j}} \right)^{f_{j,k}}, \\
 &=\prod_{j \geq 1} \left( \frac{1-v^j/k^j}{1-(uv)^j/k^j} \right)^{f_{j,k}},\\
 &= \frac{1-v}{1-uv}= \frac{1-ku}{1-ku^2}.
\end{align*}
Thus by setting $u=-u$, we have the identity
\[ \prod_{j \geq 1} (1+(-u)^j)^{-f_{j,k}} = \frac{1+ku}{1-ku^2}, \]
and we conclude that
\begin{align*}
 \sum_{n \geq 0} \frac{u^n \sum_{w \in S_n} t^{\des(w)+1} \sgn(w)}{(1-t)^{n+1}}
&= \sum_{k \geq 1} t^k \frac{1+ku}{1-ku^2} \\
&= \sum_{k \geq 1} t^k (1+ku^2 + k^2u^4 + \cdots + ku + k^2u^3 + k^3u^5+ \cdots),\\
&= \sum_{n\geq 0}\sum_{k \geq 1} u^n t^k  k^{\lceil n/2\rceil}.
\end{align*}
Taking the coefficient of $u^n$ on both sides proves Equation \eqref{eq:needprove}.
\end{proof}

We now turn to the second proof of Theorem \ref{thm:main}. For our second proof of Theorem \ref{thm:main}, we will use the generating function for the refinement of the Eulerian polynomial that gives the joint distribution of inversions and descents:
\[
A_n(q,t) = \sum_{w \in S_n} q^{\inv(w)} t^{\des(w)+1},
\]
where $\inv(w)$ is the number of inversions of $w$, defined as the number of pairs $(i,j)$ with $i<j$ and $w(i)>w(j)$.

Since the sign of $w$ is $\sgn(w)=(-1)^{\inv(w)}$, it follows that
\[ A_n(-1,t) = A_n^+(t)-A_n^-(t).\]
Loday \cite{Lo} initiated an investigation of the coefficients of $A_n(-1,t)$, which he and others, such as Desarm\'enian and Foata \cite{DF}, called ``signed Eulerian numbers.'' Among other things, Desarm\'enian and Foata prove that \cite[Theorem 1]{DF}:
\begin{equation}\label{eq:desar}
A_{2n}(-1,t) = (1-t)^n A_n(t) \quad \text{and} \quad A_{2n+1}(-1,t) = (1-t)^n A_{n+1}(t),
\end{equation}
The identities in \eqref{eq:desar} can be proved via manipulations of identities for $A_n(q,t)$ as $q \rightarrow -1$. Wachs \cite{W} also gives a combinatorial proof of \eqref{eq:desar} with a sign-reversing involution.

By adding or subtracting $A_m(t) = A_m^+(t) + A_m^-(t)$, with $m=2n$ or $m=2n+1$, to the equations in \eqref{eq:desar} we can now conclude
\begin{equation}\label{eq:ApmA}
 2 A_{2n}^{\pm}(t) = A_{2n}(t) \pm (1-t)^n A_n(t) \quad \text{and} \quad 2 A_{2n+1}^{\pm}(t) = A_{2n+1}(t) \pm (1-t)^n A_{n+1}(t).
\end{equation}
The second proof of Theorem \ref{thm:main} now follows from basic series manipulations.

\begin{proof}[Second proof of Theorem \ref{thm:main}]
Using the equations from \eqref{eq:ApmA} and the identity of Equation \eqref{eq:seriesA}, we obtain
\begin{align*}
 \frac{2A_{2n}^{\pm}(t)}{(1-t)^{2n+1}} &= \frac{A_{2n}(t)}{(1-t)^{2n+1}} \pm \frac{A_n(t)}{(1-t)^{n+1}},\\
 &=\sum_{k\geq 0} k^{2n}t^k \pm \sum_{k\geq 0} k^n t^k,\\
 &=\sum_{k\geq 0} [k^{2n}\pm k^n] t^k.
\end{align*}
and
\begin{align*}
 \frac{2A_{2n+1}^{\pm}(t)}{(1-t)^{2n+2}} &= \frac{A_{2n+1}(t)}{(1-t)^{2n+2}} \pm \frac{A_{n+1}(t)}{(1-t)^{n+2}} \\
 &= \sum_{k\geq 0} k^{2n+1} t^k \pm \sum_{k\geq 0} k^{n+1} t^k,\\
 &= \sum_{k\geq 0} [k^{2n+1} \pm k^{n+1}]t^k.
\end{align*}

Thus, for any $n \geq 1$,
\[
 \frac{A_n^{\pm}(t)}{(1-t)^{n+1}} = \sum_{k\geq 0} \frac{k^n \pm k^{\lceil n/2 \rceil}}{2} t^k,
\]
as desired.
\end{proof}

\subsection{Symmetries and recurrences}\label{sec:symmetry}

Having established our main series identity for the \pmE{} polynomials, we gather some interesting features of them now.

First, we consider the homogeneous Eulerian polynomials
\begin{align*}
 A_n(s,t) = (s^n/t)A_n(t/s) &=\sum_{w \in S_n} s^{n-1-\des(w)}t^{\des(w)} \\
 &=\sum_{w \in S_n} s^{\asc(w)}t^{\des(w)},
\end{align*}
where $\asc(w)$ denotes the number of \emph{ascents} of $w$, as mentioned in the introduction. Similarly, we let
\begin{align*}
 A_n^{\pm}(s,t) = (s^n/t)A_n^{\pm}(t/s) &=\sum_{w \in S_n^{\pm}} s^{n-1-\des(w)}t^{\des(w)} \\
 &=\sum_{w \in S_n^{\pm}} s^{\asc(w)}t^{\des(w)}.
\end{align*}
Note that while $A_n^{\pm}(t)$ has degree $n-1$ or $n$ in $t$, this homogeneous version always has degree $n-1$. For example, $A_3^+(s,t) = s^2+2st$, $A_3^-(s,t) = 2st+t^2$, $A_4^+(s,t)=s^3 + 5s^2t +5st^2 +t^3$, and $A_4^-(s,t)= 6s^2t + 6st^2$.

Because ascents and descents are swapped under the involution that reverses a permutation, $w\to \overleftarrow{w}=w(n)w(n-1)\cdots w(1)$, we have that
\[
A_n(s,t) = A_n(t,s),
\]
or equivalently,
\[
 \el{n}{k} = \el{n}{n+1-k}.
\]
That is to say, the rows of Table \ref{tab:eul} are palindromic. It is also quite well-known that the Eulerian numbers satisfy the recurrence
\begin{equation}\label{eq:eulerrec}
\el{n}{k} = (n+1-k)\el{n-1}{k-1}+ k\el{n-1}{k},
\end{equation}
with initial conditions $\el{n}{1}=\el{n}{n}=1$ for $n\geq 1$. This identity can be explained bijectively, by considering whether the insertion of the letter $n$ into a permutation
in $S_{n-1}$ leaves the number of descents unchanged or increases the number of
descents by one. See, e.g., \cite[Chapter 1]{Pet}.

The two-term recurrence of Equation \eqref{eq:eulerrec} is neatly summarized with the polynomial recurrence
\begin{equation}\label{eq:Trec}
A_{n+1}(s,t)= T\left[A_n(s,t)\right],
\end{equation}
where $T$ is the linear operator
\[
T=s+ t+ st\left( \frac{d}{ds} + \frac{d}{dt}\right).
\]
See Section 7 of Br\'and\"en's article \cite{Br} for a discussion of sequences of polynomials defined by linear operators such as these.

Now we present similar recurrences and symmetries for the \pmE{} numbers and polynomials.

\begin{prop}[Symmetries]
For any $n\geq 1$:
\begin{equation}\label{eq:symmetries}
 A_n^{\pm}(s,t) = \begin{cases} A_n^{\pm}(t,s) & \text{if } n\equiv 0, 1 \pmod 4 ,\\
               A_n^{\mp}(t,s) & \text{if } n\equiv 2,3 \pmod 4 .
               \end{cases}
\end{equation}
In terms of coefficients,
\begin{equation}\label{eq:symmetries2}
 \el{n}{k}^{\pm} = \begin{cases} \el{n}{n+1-k}^{\pm} & \text{if } n\equiv 0, 1 \pmod 4 ,\\
               \el{n}{n+1-k}^{\mp} & \text{if } n\equiv 2,3 \pmod 4 .
               \end{cases}
\end{equation}
\end{prop}

In other words, the rows of Tables \ref{tab:poseul} and \ref{tab:negeul} are palindromic for $n\equiv 0, 1 \pmod 4$, while the rows of the two tables are mirror images when $n\equiv 2, 3 \pmod 4$.

\begin{proof}
These symmetries are a consequence of the reversal involution previously discussed: $w \leftrightarrow \overleftarrow{w}$. It is straightforward to check that this map has the property that for any $w\in S_n$, $\des(\overleftarrow{w}) = n - 1 - \des(w)$, while $\inv(\overleftarrow{w}) = \binom{n}{2} - \inv(w)$.

If $n \equiv 0, 1 \pmod 4$, then $\binom{n}{2}$ is even, in which case $\inv(\overleftarrow{w}) \equiv \inv(w) \pmod 2$, and hence $\sgn(\overleftarrow{w}) =\sgn(w)$. Otherwise, if $n \equiv 2, 3 \pmod 4$, we see that $\binom{n}{2}$ is odd, and therefore $\overleftarrow{w}$ and $w$ have opposite sign. This completes the proof.
\end{proof}

To state the recurrence result, we first let
\[
 T_s = s + \frac{st}{2}\left( \frac{d}{ds} + \frac{d}{dt} \right) \quad \mbox{ and } \quad T_t = t + \frac{st}{2}\left( \frac{d}{ds} + \frac{d}{dt} \right).
\]

\begin{prop}[Recurrences]\label{prp:eulpmrec}
We have the following recurrences, for any $m\geq 1$:
\begin{align}
 A_{2m}^{\pm} &= T_s A_{2m-1}^{\pm}(s,t) + T_t A_{2m-1}^{\mp}(s,t), \label{eq:eulpolyrec1}\\
 A_{2m+1}^{\pm} &= T A_{2m}^{\pm}(s,t).\label{eq:eulpolyrec2}
\end{align}
In terms of coefficients, we have
\begin{align}
 \el{2m}{k}^{\pm} &= \el{2m-1}{k-1}^{\mp} + \left(\frac{2m-k}{2}\right)\el{2m-1}{k-1} \label{eq:eulpmrec1} \\
 & \quad +\left(\frac{k-1}{2}\right)\el{2m-1}{k}+\el{2m-1}{k}^{\pm}, \nonumber\\
 \el{2m+1}{k}^{\pm} &= (2m+2-k)\el{2m}{k-1}^{\pm}+k\el{2m}{k}^{\pm}.\label{eq:eulpmrec2}
\end{align}
\end{prop}

These recurrences can be explained combinatorially by considering the effect of inserting $n+1$ into a permutation of length $n$, much as one proves the classical Eulerian recurrence in Equation \eqref{eq:eulerrec}. However, there are many cases to check carefully; we do not know of a simple argument. The full proof of Proposition \ref{prp:eulpmrec} can be found in \cite{DS}.

\subsection{Real roots}\label{sec:roots}

In the introduction, we presented Conjecture \ref{conj:real}, which asserted that all the roots of the \pmE{} polynomials are real. It is known that the classical Eulerian polynomials are real rooted
since Frobenius; see \cite{Br} for a thorough survey of similar sequences of polynomials.

From Equation \eqref{eq:symmetries}, we find the univariate \pmE{}
polynomials satisfy $A_n^{\pm}(t) = t^{n+1} A_n^{\pm}(1/t)$ if
$n\equiv 0, 1 \pmod 4$ and $A_n^{\pm}(t) = t^{n+1} A_n^{\mp}(1/t)$ if
$n\equiv 2, 3 \pmod 4$.  This means the nonzero roots of
$A_n^{\pm}(t)$ come in reciprocal pairs when $n\equiv 0,1 \pmod 4$,
and when $n\equiv 2,3 \pmod 4$, the nonzero roots of $A_n^+(t)$ are
reciprocals of the nonzero roots of $A_n^-(t)$. It would be lovely if the
roots of $A_n^+(t)$ and $A_n^-(t)$ were interlacing. This is false in
general.

A fact one encounters in the study of real rootedness \cite{Br} is that the operator $T$ preserves real roots. Thus, by Equation \eqref{eq:eulpolyrec2}, we know that if $A_{2m}^{\pm}(t)$ has real roots, then so does $A_{2m+1}^{\pm}(t)$. Therefore Conjecture \ref{conj:real} only needs to be proved in the even case. To this end, we can apply the recurrences of Proposition \ref{prp:eulpmrec} twice in a row to see
\[
A_{2m+2}^{\pm}(s,t) =T_sT A_{2m}^{\pm}(s,t) +T_tT A_{2m}^{\mp}(s,t).
\]
It is not difficult to prove that the operators $T_s$ and $T_t$ preserve real-rootedness, and empirical evidence suggests that the pair of polynomials on the right are real and interlacing as well.

\section{Central limit theorems}\label{sec:clt}

In this section we prove Theorem \ref{thm:clt}. In fact we will give two proofs of this central limit theorem, one using ``analytic combinatorics'' starting from the identity in Theorem \ref{thm:main}, and the other using the ``method of moments'' in comparison with the usual Eulerian distribution.

For our first proof we start with the identity of Theorem \ref{thm:main} and use a  modified Curtiss' theorem proved in Kim and Lee's paper \cite[Proposition 2.2]{KL}.  We state a version of that result in the following lemma.

\begin{lemma}\label{lem:KL}
Let $X$ be a probability distribution and let $\{X_n\}$ be a sequence of random variables, with moment generating functions $M_X$ and $M_{X_n}$, respectively. Suppose there is an open interval $I \subseteq \RR$ on which
\[
 \lim_{n\to \infty} M_{X_n}(s) = M_X(s)
\]
for all $s \in I$. Then $X_n$ converges in distribution to $X$.
\end{lemma}

We can express the moment generating function for the \pmE{} numbers $\el{n}{k}^{\pm}$ explicitly. To be precise, we introduce a random variable $W_n^{\pm}$ which takes values in $\{1,\cdots,n\}$ with probabilities
\[
    \mathbb{P}(W_{n}^{\pm} = k) = \frac{1}{|S_n^{\pm}|}\el{n}{k}^{\pm}.
\]
Note that $W_n^{\pm}$ is distributed as the descent plus one of a random permutation drawn uniformly at random from $S_n^{\pm}$. Then, from Theorem \ref{thm:main}, the moment generating function of~$W_n^{\pm}$ takes the form
\[
    \mathbb{E}[\exp\{s W_{n}^{\pm}\}]
    = \frac{A_n^{\pm}\left( e^s \right)}{n!/2}
    = \frac{\left( 1 - e^s \right)^{n+1}}{n!} \sum_{k=1}^\infty \left( k^n \pm k^{\lceil n/2 \rceil} \right) e^{ks}.
\]
Our task now shifts to estimating this function well enough to show that for all fixed $s$ in an appropriately chosen interval $I$, the normalized moment generating function converges as $n\to \infty$ to a moment generating function $M_X(s)$ for a normal distribution $X$ with the desired mean and variance.

\begin{proof}[First proof of Theorem \ref{thm:clt}]
We argue for positive Eulerian numbers; the argument for negative Eulerian numbers is almost identical. Fix a real number $s>0$.

The asymptotic normality of $W_n^{+}$ translates to the claim that the following normalized random variable
\[
    Z_{n}^{+} = \frac{1}{\sqrt{n+1}}\left( W_{n}^{+} - \frac{n+1}{2}\right)
\]
converges in distribution to the normal distribution with zero mean and variance $\frac{1}{12}$. In view of Lemma 3.1 and the subsequent remark, it is sufficient to prove that the Laplace transform $\mathbb{E}[\exp\{-s Z_{n}^{+}\}]$ converges to $\exp\{\frac{s^2}{24}\}$ as $n\to\infty$ for each $s$ in the interval $I = (0, \infty)$. We henceforth write $M_{n}^{+}(s) = \mathbb{E}[\exp\{-s Z_{n}^{+}\}]$. Then
\[
    \begin{aligned}
        M_n^+(s)
        &= e^{\frac{s}{2}\sqrt{n+1}} \, \mathbb{E}\big[ \exp\left\{-sW_n^{+}/\smash{\sqrt{n+1}}\right\} \big] \\
        &= \frac{e^{\frac{s}{2}\sqrt{n+1}} \big( 1 - e^{-s/\sqrt{n+1}} \big)^{n+1}}{n!} \sum_{k=1}^\infty \left( k^n + k^{\lceil n/2 \rceil} \right) e^{-sk/\sqrt{n+1}}.
    \end{aligned}
\]
Plugging $x = \frac{s}{\sqrt{n+1}}$ in to the approximation
\[
    \frac{1 - e^{-x}}{x}
    = \exp\left\{ - \frac{x}{2} + \log \left( \frac{\sinh(x/2)}{x/2} \right) \right\}
    = \exp \left\{ -\frac{x}{2} + \frac{x^2}{24} + \mathcal{O}(x^3) \right\},
\]
we see that the prefactor of $M_n^+(s)$ equals
\[
    \frac{e^{\frac{s}{2}\sqrt{n+1}} \big( 1 - e^{-s/\sqrt{n+1}} \big)^{n+1}}{n!}
    = e^{\frac{s^2}{24} + \mathcal{O}( n^{-1/2} )} \frac{1}{n!} \left( \frac{s}{\sqrt{n+1}} \right)^{n+1}.
\]
In terms of $M_n^{+}(s)$ itself,
\[
	M_n^+(s) = e^{\frac{s^2}{24} + \mathcal{O}( n^{-1/2} )} \frac{1}{n!} \left( \frac{s}{\sqrt{n+1}} \right)^{n+1} \sum_{k=1}^\infty \left( k^n + k^{\lceil n/2 \rceil} \right)e^{-sk/\sqrt{n+1}}.
\]

Given this representation, we estimate the summation part via approximation to integrals. Indeed, for $t \in (0,1)$ and $a \geq 1$, we have
\[
	\int_{a-1}^{a} x^n t^{x+1} \, dx \leq a^n t^a \leq \int_a^{a+1} x^n t^{x-1} \, dx.
\]
Summing these over all $a \in \{1, 2, \cdots\}$, we obtain
\[
	t \int_{0}^{\infty} x^n t^x \, dx \leq \sum_{a=1}^\infty a^n t^a \leq \int_{1}^\infty x^n t^{x-1} \, dx \leq \frac{1}{t} \int_0^\infty x^n t^x \, dx.
\]
Since $s>0$, we have $0<e^{-s/\sqrt{n}}<1$, and setting $t = e^{-s/\sqrt{n+1}}$, we find that both $t$ and $\frac{1}{t}$ have asymptotic form $1 + \mathcal{O}(n^{-1/2})$. Thus we can replace the sum with an integral and our estimate now becomes
\[
	M_n^+(s) = e^{\frac{s^2}{24} + \mathcal{O}( n^{-1/2} )} \frac{1}{n!} \left( \frac{s}{\sqrt{n+1}} \right)^{n+1}  \int_{0}^{\infty} \left( x^n + x^{\lceil n/2 \rceil} \right) e^{-sx/\sqrt{n+1}} \, dx.
\]
Making the substitution $u = sx/\sqrt{n+1}$ and using some easy comparisons, we find
\begin{align*}
	M_n^+(s)
	&= e^{\frac{s^2}{24} + \mathcal{O}( n^{-1/2} )} \frac{1}{n!} \int_{0}^{\infty} \bigg( u^n + \left( \frac{s}{\sqrt{n+1}} \right)^{\lfloor n/2 \rfloor} u^{\lceil n/2 \rceil} \bigg) e^{-u} \, du, \\
	&= e^{\frac{s^2}{24} + \mathcal{O}( n^{-1/2} )} \frac{1}{n!} \bigg( n! + \left( \frac{s}{\sqrt{n}} \right)^{\lfloor n/2 \rfloor} \left( \left\lceil \frac{n}{2} \right\rceil \right)! \bigg), \\
	&= e^{\frac{s^2}{24} + \mathcal{O}( n^{-1/2} )}
\end{align*}
Therefore, the convergence $M_n^{+}(s) \to \exp\{ \frac{s^2}{24} \}$ holds for all $s$ in the interval $I = (0,\infty)$, and we are done.
\end{proof}

Our second proof relies on the following result about moments of the \pmE{} distributions. We note that Section 9.3 of \cite{HCD} also proves Theorem \ref{thm:clt} by moments.

\begin{prop} \label{prop:mommatch}
Let $r$ be any positive integer. Then for $\lfloor n/2 \rfloor > r$, the $r$th moment of the \pmE{} distribution equals the $r$th moment of the Eulerian distribution.
\end{prop}

\begin{proof}
Instead of working with the $r$th moment, we can work with the $r$th falling moment,
which can be computed from the relevant generating function by differentiating (with
respect to $t$) $r$ times and then setting $t=1$.

 Restating Equation \eqref{eq:ApmA}, we have
\[
A_n^{\pm}(t)  = \frac{1}{2} A_n(t) \pm \frac{1}{2} (1-t)^{\lfloor n/2 \rfloor} A_{\lceil n/2 \rceil}(t).
\]
 Now observe that if $\lfloor n/2 \rfloor > r$, then differentiating
 \[(1-t)^{\lfloor n/2 \rfloor}\]
 $r$ times and setting $t=1$ gives $0$. And by the product rule, differentiating the expression
 \[(1-t)^{\lfloor n/2 \rfloor} A_{\lceil n/2 \rceil}(t)\]
 $r$ times and setting $t=1$ also gives $0$. The theorem follows.
\end{proof}

It is a well-known fact that the Eulerian numbers $\el{n}{k}$ are asymptotically normal with mean $(n+1)/2$ and variance $(n+1)/12$. See, e.g., Bender \cite{B}. Then by taking $r=1$ in Proposition \ref{prop:mommatch}, we find that for $n \geq 4$, averaging the number of descents over positive signed permutations gives $(n-1)/2$, and that averaging the number of descents over negative signed permutations also gives $(n-1)/2$. Likewise, if $r=2$, we find in both cases a variance of $(n+1)/12$ when $n\geq 6$. This gives our second proof of Theorem \ref{thm:clt}.

\begin{proof}[Second proof of Theorem \ref{thm:clt}]
This is immediate from the method of moments, Proposition \ref{prop:mommatch}, and the fact that the Eulerian distribution is asymptotically normal with the claimed mean and variance.
\end{proof}

We close this section by mentioning that if the polynomials $A_n^{\pm}(t)$ have only real roots, i.e., if Conjecture \ref{conj:real} holds, then Harper's method would give a third proof of Theorem \ref{thm:clt}.

\section{Shuffling and sign}\label{sec:shuffle}

The mathematics of card shuffling is a lively topic; see \cite{Dsurv} for a nice survey. In this section we study the sign of a permutation generated by the Gilbert-Shannon-Reeds (GSR) model of riffle shuffling. We derive a simple and striking exact formula for the chance of a given sign after a riffle shuffle, and this formula
implies that (for large $n$), one shuffle suffices to randomize the sign of a permutation.

To begin we give some background on the GSR model of riffle shuffling, definitively
studied in a lovely paper of Bayer and Diaconis \cite{BD}, to which we refer the
reader for further background.

An \emph{$a$-shuffle} is defined as follows. Choose integers $j_1,\ldots,j_a$ according to the multinomial distribution,
\[ P(j_1,\cdots,j_a) = \binom{n}{ j_1, \ldots, j_a} / a^n .\]
Thus $0 \leq j_i \leq n$, $\sum_{i=1}^a j_i=n$, and the $j_i$ have the same distribution as the number of balls in each box $i$ if $n$ balls are dropped at random into $a$ boxes.

Given the $j_i$, cut off the top $j_1$ cards, the next $j_2$ cards and so on,
producing $a$ packets (some possibly empty). Then drop cards one at a time, according to the rule that if there are $A_j$ cards in packet $j$, the next card is dropped from packet $i$ with probability $A_i/(A_1+\cdots+A_a)$. This is done until all cards have been dropped.

When $a=2$ this is a realistic model for how people shuffle cards. It turns out that an $a_1$-shuffle followed by an $a_2$-shuffle is equivalent to an $(a_1 a_2)$-shuffle. Thus $k$ iterations of a $2$-shuffle is equivalent to a single $2^k$-shuffle.

In what follows, we let $P_{n,a}(w)$ denote the probability of a permutation $w$ after an $a$-shuffle started from the identity. In fact there is an explicit formula for this quantity, derived in \cite{BD}:
\begin{equation} \label{eq:shufform}
P_{n,a}(w) = {n+a - \des(w^{-1}) - 1 \choose n}/a^n.
\end{equation}
In particular, Equation \eqref{eq:shufform} connects riffle shuffling with descents, showing that this probability depends only on the number of descents of the inverse permutation.

The main purpose of this section is to prove Theorem \ref{thm:shuffle}, which establishes a formula for the probability of having a permutation with sign $1$ after $k$ iterations of $a$-shuffling starting from the identity permutation $123\cdots n$. Letting $P_{n,a^k}^+$ denote this probability (and $P_{n,a^k}^-$ its complementary probability), we will show
\begin{equation}\label{eq:shuffleprob}
P_{n,a^k}^+= \frac{1}{2} + \frac{1}{2a^{k \lfloor n/2 \rfloor}} .
\end{equation}

We will give two proofs of this fact, and sketch a third using unpublished work of Amy Pang. One of these again uses the power series identity in Lemma \ref{lem:usesim} as we did in Section \ref{sec:mainproof}. The other builds off of the identity in \eqref{eq:seriesApm} from Theorem \ref{thm:main}.

To get the first proof started, we recall the following series identity from Diaconis, McGrath, and Pitman \cite[Proposition 5.6]{DMP}:
\begin{equation} \label{eq:DMP}
1 + \sum_{n \geq 1} u^n \sum_{w \in S_n} P_{n,a}(w) \prod_{i \geq 1} x_i^{n_i(w)}
= \prod_{j \geq 1} (1-u^jx_j/a^j)^{-f_{j,a}},
\end{equation}
which is very similar to Fulman's identity \eqref{eq:F1} we used in Section \ref{sec:mainproof}. As before the quantity $n_i(w)$ counts the number of $i$-cycles of $w$ and $f_{j,a}$ counts the number of primitive necklaces of length $j$ on an alphabet of $a$ letters. Also as in Section \ref{sec:mainproof}, we recall that $\sgn(w) = (-1)^{n+c(w)}$, so that replacing $u$ with $-u$ and setting the $x_i$ equal to $(-1)$ yields:
\begin{equation}\label{eq:DMPspecial}
1+ \sum_{n \geq 1} u^n \sum_{w \in S_n} P_{n,a}(w) \sgn(w)
 = \prod_{j \geq 1} (1+(-u)^j/a^j)^{-f_{j,a}},
\end{equation}
which is our jumping off point for what follows.

\begin{proof}[First proof of Theorem \ref{thm:shuffle}]
The coefficient of $u^n$ on the left hand side of Equation \eqref{eq:DMPspecial} is
\[
 \sum_{w \in S_n^+} P_{n,a}(w) - \sum_{w' \in S_n^-} P_{n,a}(w') = P_{n,a}^+-P_{n,a}^-,
\]
i.e., the difference in probabilities between positively and negatively signed permutations.

By application of Lemma \ref{lem:usesim} as in the first proof of Theorem \ref{thm:main}, we can show the right hand side of Equation \eqref{eq:DMPspecial} is equal to
\[ \frac{1+u}{1-u^2/a},\]
and taking the coefficient of $u^n$ in this series gives $1/a^{\lfloor n/2 \rfloor}$.

We have shown that
\[ P_{n,a}^+ - P_{n,a}^- = \frac{1}{a^{\lfloor n/2 \rfloor}} .\]
Since $P_{n,a}^+ + P_{n,a}^- = 1$, it follows that
\[ P_{n,a}^+ = \frac{1}{2} + \frac{1}{2a^{\lfloor n/2 \rfloor}}.\]
Since $k$ $a$-shuffles is the same as one $a^k$ shuffle, the result follows.
\end{proof}

Our second proof builds from the identity in Theorem \ref{thm:main}.

\begin{proof}[Second proof of Theorem \ref{thm:shuffle}]
Equation \eqref{eq:seriesApm} gives
\[
\frac{1}{(1-t)^{n+1}} \sum_{w \in S_n^+} t^{\des(w)+1} =
\sum_{a \geq 1} \frac{a^n + a^{\lceil n/2 \rceil}}{2} t^a .
\]
Taking the coefficient of $t^a$ on both sides gives
\[
\sum_{w \in S_n^+ } \binom{n+a-\des(w)-1}{ n} = \frac{a^n + a^{\lceil n/2 \rceil}}{2} .
\]
Since $\sgn(w)=\sgn(w^{-1})$, therefore,
\[
\sum_{w \in S_n^+} {n+a-\des(w^{-1})-1 \choose n} = \frac{a^n + a^{\lceil n/2 \rceil}}{2} .
\]
Hence by the formula for $P_{n,a}(w)$ in Equation \eqref{eq:shufform}, we have
\[
\sum_{w \in S_n^+} P_{n,a}(w) = \frac{1}{a^n} \left[  \frac{a^n + a^{\lceil n/2 \rceil}}{2}  \right]
= \frac{1}{2} + \frac{1}{2a^{\lfloor n/2 \rfloor}},\]
as desired.
\end{proof}

We now sketch a third proof of Theorem \ref{thm:shuffle}. As we learned from Diaconis,
this is a simple consequence of
a result of Amy Pang (unpublished and prior to our work) stating that $f(w)=\sgn(w)$ is a right eigenfunction of the $a$-shuffle Markov chain with eigenvalue $a^{- \lfloor n/2 \rfloor}$.
The proof of Pang's result uses a heavy dose of Hopf algebras, along the lines of \cite{DPR}.

\begin{proof}[Third proof of Theorem \ref{thm:shuffle}]
Letting $P$ denote the transition operator of the riffle shuffling Markov chain, one has that $P$ acts on functions $f$ by $P(f)[x] = \sum_{y\in S_n} P(x,y) f(y)$, and more generally,
\[ P^k(f)[x]= \sum_{y\in S_n} P^k(x,y) f(y).\]

Letting $\id$ denote the identity permutation, and taking $f$ to be the sign function, since $f=\sgn$ is a right eigenfunction with eigenvalue $a^{- \lfloor n/2 \rfloor}$, it follows that
\[
P^k \sgn [\id] = \frac{\sgn(\id)}{a^{k \lfloor n/2 \rfloor}} = \frac{1}{a^{k \lfloor n/2 \rfloor}}.
\]
On the other hand,
\[
P^k \sgn [\id] = \sum_{y \in S_n} P^k(\id,y) \sgn(y),
\]
which is simply the chance of sign $1$ after $k$ steps minus the chance of sign $-1$ after $k$ steps. So this difference is equal to $1/(a^{k \lfloor n/2 \rfloor})$, and the result follows as in the first proof.
\end{proof}

We finish this section with two further remarks on shuffling and sign.

\begin{rem}[Time to randomness]
It is well known \cite{BD} that order $\frac{3}{2} \log_2(n)$ many $2$-shuffles are necessary and sufficient to mix a deck of $n$ cards. However if one is only interested in certain features of the deck (for instance the number of fixed points of a permutation), then randomness can occur much earlier. Theorem \ref{thm:shuffle} shows that (for large $n$), the sign of a permutation becomes random after one step.
\end{rem}

\begin{rem}[Shelf-shuffling machines and sign]
Another model for card shuffling is the shelf-shuffling machine as studied in \cite{DFH}. These machines are used in casinos with $m=10$. One of the main findings of \cite{DFH} is that a single use of a shelf-shuffler does {\it not} adequately mix $52$ cards, but that two iterations {\it do} adequately mix $52$ cards. One may wonder what effect shelf-shuffling machines have on the sign of a permutation. It turns out that the probability that $\sgn(w)=1$ after one pass through a shelf-shuffling machine is \emph{exactly} 1/2 for all $n\geq 2$.

This follows because, as found in \cite{DFH}, the probability of obtaining the permutation $w \in S_n$ when using a shelf shuffler with $m$ shelves
and an $n$ card deck is
\[
\frac{4^{\val(w)+1}}{2 (2m)^n} \sum_{a=0}^{m-1} {n+m-a-1 \choose n}{n-1-2\val(w) \choose a-\val(w)},
\]
where $\val(w)$ is the number of \emph{valleys} $w(i-1)>w(i)<w(i+1)$. That is, the probability of reaching a particular permutation depends only on the number of valleys in that permutation.

However, for any number of valleys, say $k$, we have an equal number of positively and negatively signed permutations with that many valleys, i.e.,
\[
 |\{ w \in S_n^+ : \val(w) = k \}| = |\{ w' \in S_n^- : \val(w') = k \}|.
\]
This is easily seen under the involution that swaps $n$ and $n-1$; neither of these elements can be in a valley and swapping them only changes the number of inversions by one. For example, $w =34812765  \leftrightarrow 34712865=w'$.
\end{rem}

\section{Analogous results for the hyperoctahedral group}\label{sec:Bn}

We now present our results for the hyperoctahedral group. For more background on the combinatorics of Coxeter groups, see \cite{BjBr} or \cite[Part III]{Pet}.

First, we define elements of the hyperoctahedral group $B_n$ to be permutations of $[\pm n] = \{-n,\ldots,-1,0,1,\ldots,n\}$ with signed symmetry. That is, $w \in B_n$ if and only if $w$ is a bijection $[\pm n] \to [\pm n]$ such that $w(-i) = -w(i)$ for all $i$. This symmetry requirement means such a permutation is determined by the entries $w(1), \ldots, w(n)$, and we write elements in one-line notation as $w = w(1)\cdots w(n)$, using bars to indicate negative entries. For example, $w = \bar{1} 3 2 \in B_3$ is the bijection given by $w(-3)=-2$, $w(-2)=-3$, $w(-1)=1$, $w(0)=0$, $w(1)=-1$, $w(2)=3$, and $w(3)=2$.

To any permutation $w \in S_n$ we can associate $2^n$ elements of $B_n$ by negating any subset of the entries $w(1),\ldots,w(n)$, and so there are $2^n n!$ elements of $B_n$, commonly, if somewhat confusingly for this context, called \emph{signed permutations}. This notion of decorating ordinary permutations with bars is important for us. So for any element $u \in S_n$ and any $J \subseteq [n]$, let $u^J=w \in B_n$ denote the element given by
\[
 w(j) = \begin{cases} u(j) & j \notin J \\ -u(j) & j \in J \end{cases}
\]
For example, $134652^{\{1,4,5\}} = \bar 1 34 \bar 6 \bar 5 2$.

The definitions of descent and inversion number for elements of the hyperoctahedral group are motivated by the theory of Coxeter groups as discussed in \cite{BjBr} and \cite[Part III]{Pet}. For any $w \in B_n$, we say $i$ is a \emph{type B descent} if $w(i) > w(i+1)$, for $i=0,1,\ldots,n-1$. We let $\des_B(w)$ denote the number of type B descents:
\[\des_B(w) = |\{ 0\leq i \leq n-1 : w(i) > w(i+1)\}|.\]
For example, $\des_B(\bar{1} 3 2) = 2$, because $w(0)>w(1)$ and $w(2)>w(3)$.

\begin{rem}
We mention that Reiner's papers \cite{Rei}, \cite{Rei2} use a different notion of descents. Using the ordering
\[ 1 < 2 < \cdots < n < -n < \cdots < -2 < -1, \]
he says that $w$ has a descent at position $i$ ($1 \leq i \leq n-1$) if $w(i) > w(i+1)$ and has a descent at position $n$ if $w(n)<0$. While these definitions are different, it is straightforward to check that letting $w_0 = n(n-1) \cdots 1$, then a signed permutation $w \in B_n$ has $d$ descents with Reiner's definition if and only if $w_0 w w_0$ has $d$ descents with our definition. Moreover, since $w_0$ is an involution, replacing $w$ by $w_0 w w_0$ leaves the signed cycle type (defined below) of $w$ unchanged.
\end{rem}

The type B Eulerian numbers are denoted $\el{B_n}{k}$, $1\leq k \leq n$, and they count the number of $w\in B_n$ with $k$ descents, i.e.,
\[
 \el{B_n}{k} = |\{ w \in B_n : \des_B(w) = k\}|.
\]
The first few rows of the type B Eulerian triangle are shown in Table \ref{tab:Beul}.

The type B \emph{Eulerian polynomials} are the generating functions for the rows of this triangle, i.e., the generating function for signed permutations according to the type B descent statistic:
\[
 B_n(t) = \sum_{k=0}^n \el{B_n}{k} t^k =  \sum_{ w \in B_n} t^{\des_B(w)}.
\]
For example, we have $B_2(t) = 1+6t + t^2$ and $B_3(t) = 1 + 23t + 23t^2 + t^3$.

\begin{table}[t]
\begin{center}
\begin{tabular}{c |rrrrrrrr}
$n\backslash k$ & 0 & 1 & 2 & 3 & 4 &  5 & 6 \\
\hline
1 & 1 & 1 \\
2 & 1 & 6 & 1 \\
3 & 1 & 23 & 23 & 1 \\
4 & 1 & 76 & 230 & 76 & 1 \\
5 & 1 & 237 & 1682 & 1682 & 237 & 1  \\
6 & 1 & 722 & 10543 & 23548 & 10543 & 722 & 1\\
\end{tabular}
\end{center}
\caption{Triangle of the type B Eulerian numbers $\el{B_n}{k}$, the number of elements in $B_n$ with $k$ descents.}\label{tab:Beul}
\end{table}

The \emph{sign} of an element $w \in B_n$ is again motivated by the theory of Coxeter groups, but from a combinatorial standpoint it can be phrased very simply in terms of the number of negative entries of $w$ and the sign of the underlying unsigned permutation. That is, if $w = u^J$ for $u \in S_n$ and $J \subseteq [n]$,
\begin{equation}\label{eq:sgnB}
 \sgn_B(w) = (-1)^{|J|}\sgn(u).
\end{equation}
For example, $u=3412 \in S_4$ has sign $+1$, so $w = \bar 3 4 \bar 1 \bar 2$ has sign $(-1)^3\cdot 1 = -1$. As with the symmetric group, we denote the sets of elements with positive and negative sign by $B_n^+$ and $B_n^-$, respectively.

We now define the type B variants of the positive and negative Eulerian numbers:
\[
 \el{B_n}{k}^+ = |\{w \in B_n^+ : \des_B(w) = k\}| \quad \mbox{and} \quad \el{B_n}{k}^- = |\{w \in B_n^- : \des_B(w) = k\}|.
\]
The first few rows of the type B positive and negative Eulerian numbers are shown in Tables \ref{tab:Bposeul} and \ref{tab:Bnegeul}.

\begin{table}[t]
\begin{center}
\begin{tabular}{c |rrrrrrrr}
$n\backslash k$ & 0 & 1 & 2 & 3 & 4 &  5 & 6 \\
\hline
1 & 1 & 0 \\
2 & 1 & 2 & 1 \\
3 & 1 & 10 & 13 & 0\\
4 & 1 & 36 & 118 & 36 & 1 \\
5 & 1 & 116 & 846 & 836 & 121 & 0\\
6 & 1 & 358 & 5279 & 11764 & 5279 & 358 & 1  \\
\end{tabular}
\end{center}
\caption{Triangle of the type B positive Eulerian numbers, $\el{B_n}{k}^+$, the number of elements in $B_n^+$ with $k$ descents.}\label{tab:Bposeul}
\end{table}

\begin{table}[t]
\begin{center}
\begin{tabular}{c |rrrrrrrr}
$n\backslash k$ & 0 & 1 & 2 & 3 & 4 &  5 & 6 \\
\hline
1 & 0 & 1\\
2 & 0 & 4 & 0 \\
3 & 0 & 13 & 10 & 1 \\
4 & 0 & 40 & 112 & 40 & 0\\
5 & 0 & 121 & 836 & 846 & 116 & 1 \\
6 & 0 & 364 & 5264 & 11784 & 5264 & 364 & 0  \\
\end{tabular}
\end{center}
\caption{Triangle of the negative Eulerian numbers, $\el{B_n}{k}^-$, the number of elements in $B_n^-$ with $k$ descents.}\label{tab:Bnegeul}
\end{table}

Similarly, we define the type B \pmE{} polynomials:
\[
 B_n^+(t) = \sum_{k=0}^n \el{B_n}{k}^+ t^k =  \sum_{ w \in B_n^+} t^{\des_B(w)},
\]
and
\[
B_n^-(t) = \sum_{k=0}^n \el{B_n}{k}^- t^k =  \sum_{ w \in B_n^-} t^{\des_B(w)}.
\]
For example, $B_2^+(t) = 1+2t+t^2$, $B_2^-(t) = 4t$, $B_3^+(t) = 1+10t+13t^2$, and $B_3^-(t) = 13t+10t^2 + t^3$.

\subsection{A type B generating function identity}

The analogue of Equation \eqref{eq:seriesA} for the hyperoctahedral group is:
\begin{equation}\label{eq:seriesB}
 \frac{B_n(t)}{(1-t)^{n+1}} = \sum_{k\geq 0} (2k+1)^n t^k.
\end{equation}
See, for example, work of Brenti \cite[Theorem 3.4]{Brenti}.

The analogous identity for the type B \pmE{} polynomials is as follows.

\begin{thm}[Type B generating function identity]\label{thm:Bmain}
 For all $n \geq 1$,
 \begin{equation}\label{eq:seriesBpm}
 \frac{B_n^{\pm}(t)}{(1-t)^{n+1}} = \sum_{k\geq 0} \frac{(2k+1)^n\pm 1}{2} t^k.
 \end{equation}
\end{thm}

As in the symmetric group case, our first proof of Theorem \ref{thm:Bmain} uses the joint distribution of descents and (signed) cycle type, in this case relying on a result of Reiner \cite{Rei2}.
Given a cycle
\[
\left( \begin{array}{ccccc}
j_1 & j_2 & \cdots & j_{i-1} & j_i\\
\epsilon_1 j_2 & \epsilon_2 j_3 & \cdots & \epsilon_{i-1} j_i & \epsilon_{i} j_1
\end{array} \right),
\]
where $\epsilon_i = \pm 1$ and $j_i > 0$, we say that $C$ is a positive cycle (of size $i$) if $\epsilon_1 \epsilon_2 \cdots \epsilon_i = +1$ and a negative cycle (of size $i$) if $\epsilon_1 \epsilon_2 \cdots \epsilon_i = -1$. Given an element $w$ of $B_n$, we let $n_i(w)$ denote the number of positive cycles of $w$ of size $i$, and we let $m_i(w)$ denote the number of negative cycles of $w$ of size $i$.

The following theorem is a special case of Theorem 4.1 of Reiner \cite{Rei2}. Compare with the type A identity in Equation \eqref{eq:F1}.

\begin{thm}[Reiner] \label{vics}
We have the following series identity:
\[ \sum_{n \geq 0} \frac{u^n \sum_{w \in B_n} t^{\des_B(w)+1} \prod_i x_i^{n_i(w)} y_i^{m_i(w)}}{(1-t)^{n+1}} = 1+ \sum_{k \geq 1} t^k \frac{1}{1-x_1 u} \prod_{m > 1} \left( \frac{1+y_m u^m}
{1-x_m u^m} \right)^{N^*(2k-1,2m)}, \]
where
\[
N^*(2k-1,2m) :=
\frac{1}{2m} \sum_{d|m \atop d {\tiny \textrm{ odd}} }
\mu(d) [(2k-1)^{m/d}-1].
\]
\end{thm}

Given this theorem, we are ready to prove Theorem \ref{thm:Bmain}.

\begin{proof}[First proof of Theorem \ref{thm:Bmain}.]
It is not difficult to see that
\[
\sgn_B(w) = (-1)^{n - \sum_i n_i(w)}.
\]
Setting $u \mapsto -u, x_i \mapsto -1, y_i \mapsto 1$ in Theorem \ref{vics} gives that
\[
\sum_{n \geq 0} \frac{u^n \sum_{w \in B_n} \sgn_B(w) t^{\des_B(w)+1}}{(1-t)^{n+1}} = 1+\sum_{k \geq 1} \frac{t^k}{(1-u)} .
\]

Taking the coefficient of $u^n$, $n>0$, gives that
\[
\frac{\sum_{w \in B_n} \sgn_B(w) t^{\des_B(w)+1}}{(1-t)^{n+1}}= \sum_{k \geq 1} t^k,
\]
and multiplying by $(1-t)^{n+1}/t$ yields
\[
 \sum_{w \in B_n} \sgn_B(w) t^{\des_B(w)} = (1-t)^{n+1}\sum_{k\geq 0} t^k.
\]
The result follows by summing this identity with \eqref{eq:seriesB}:
\[ \sum_{w \in B_n} t^{\des_B(w)} = (1-t)^{n+1} \sum_{k \geq 0} (2k+1)^n t^k,\]
and noting that
\[
 2B_n^{\pm}(t) = \left( \sum_{w \in B_n} t^{\des_B(w)} \right) \pm \left(\sum_{w \in B_n} \sgn_B(w) t^{\des_B(w)} \right).
\]
\end{proof}

For the second proof of Theorem \ref{thm:Bmain}, we first step back and define the polynomial
\[
 B_n(q,t) = \sum_{w \in B_n} q^{\inv_B(w)} t^{\des_B(w)},
\]
where $\inv_B(w)$ is the number of type B inversions, defined as follows:
\begin{align*}
 \inv_B(w) &= |\{ 1\leq i < j \leq n : w(i) > w(j)\}| + |\{ 1\leq i < j \leq n : -w(i) > w(j)\}| + \\
  &\quad |\{ 1 \leq i \leq n : w(i) < 0\}|.
\end{align*}
Both $\inv(w)$ and $\inv_B(w)$ are combinatorial characterizations of the Coxeter length of an element. See \cite{BjBr} for more background.

Using the framework of Chapter 7 of \cite{BjBr} and work of Tan \cite[Theorem C]{TanLin} for type B, it follows that at $q=-1$ we have
\begin{equation}\label{eq:Bneasy}
 B_n(-1,t) = (1-t)^n.
\end{equation}
This equation is also in Reiner's paper \cite{Rei}, where the identity is given a combinatorial explanation with a sign-reversing involution on $B_n$.

\begin{proof}[Second proof of Theorem \ref{thm:Bmain}]
The theorem follows quickly from identity \eqref{eq:Bneasy}. Since
\[
B_n^+(t) - B_n^-(t) =B_n(-1,t) = (1-t)^n,
\]
and
\[
B_n^+(t) + B_n^-(t) = B_n(t),
\]
we find
\begin{equation}\label{eq:moments}
2B_n^{\pm}(t) = B_n(t) \pm (1-t)^n.
\end{equation}
Dividing both sides by $2(1-t)^{n+1}$ and using the series identity for $B_n(t)$ in \eqref{eq:seriesB}, we get
\[
\frac{B_n^{\pm}(t)}{(1-t)^{n+1}} = \frac{B_n(t) \pm (1-t)^n}{2(1-t)^{n+1}} = \sum_{k\geq 0} \frac{ (2k+1)^n \pm 1}{2} t^k,
\]
as desired.
\end{proof}

\subsection{Type B symmetries and recurrences}

As we did with the symmetric group, we now consider homogenized polynomials
\[
 B_n(s,t) = s^nB_n(t/s) = \sum_{w \in B_n} s^{\asc_B(w)} t^{\des_B(w)},
\]
where $\asc_B(w) = n- \des_B(w)$ counts the number of type B ascents, and likewise for the polynomials $B_n^{\pm}(s,t)$.

In the hyperoctahedral group, a nice involution comes from negating all entries, i.e., $w \to \overline{w}$, where $\overline{w}(i) = -w(i)$. This involution clearly satisfies $w(i)>w(i+1)$ if and only if $\overline{w}(i) < \overline{w}(i+1)$, $i=0,1,\ldots,n$, so
\[
 B_n(s,t) = B_n(t,s).
\]

From Equation (\ref{eq:sgnB}), we see that for any $w = u^J$,
\[
\sgn_B(\overline{w})\cdot \sgn_B(w) = (-1)^{n-|J|}\sgn(u)\cdot (-1)^{|J|}\sgn(u) = (-1)^n,
\]
and therefore,
\[
 \sgn_B(\overline{w}) = \begin{cases} \sgn_B(w) & \mbox{ if $n$ even,}\\ -\sgn_B(w) & \mbox{ if $n$ odd.} \end{cases}
\]
We get the following symmetries as a consequence.

\begin{prop}[Type B symmetries]
For any $n\geq 1$:
\[
 B_n^{\pm}(s,t) = \begin{cases} B_n^{\pm}(t,s) & \mbox{ if $n$ even,}\\ B_n^{\mp}(t,s) & \mbox{ if $n$ odd.} \end{cases}
\]
\end{prop}

Though usually referenced in an equivalent form for the univariate polynomials (e.g., in \cite[Theorem 3.4(i)]{Brenti}), these polynomials satisfy
\[
 B_{n+1}(s,t) = U[B_n(s,t)],
\]
where $U$ is the linear operator
\[
 U = s+t +2st\left( \frac{d}{ds} + \frac{d}{dt} \right) = T+st\left( \frac{d}{ds} + \frac{d}{dt} \right),
\]
and $T$ is the operator from Section \ref{sec:symmetry}. To prove the recurrence, the idea is to imagine the effect of inserting the letter $(n+1)$ or $\overline{(n+1)}$ into an element of $B_n$ written in one-line notation.

This logic extends to the polynomials $B_n^{\pm}(s,t)$, though there are more delicate cases to consider. Details can be found in \cite[Section 7]{DS}.

\begin{prop}[Type B recurrences]\label{typeBrec}
For any $n\geq 1$,
\begin{align*}
 B_{n+1}^{\pm}(s,t) &= sB_n^{\pm}(s,t) + tB_n^{\mp}(s,t) + st\left( \frac{d}{ds} + \frac{d}{dt} \right)B_n(s,t),\\
  &= T_s B_n^{\pm}(s,t) + T_tB_n^{\mp}(s,t),
\end{align*}
where operators $T_s$ and $T_t$ are as in Proposition \ref{prp:eulpmrec}.
\end{prop}

The recurrence here is similar (and simpler) than that for type A in Proposition \ref{prp:eulpmrec}. We conjecture that the polynomials $B_n^{\pm}(t)$ are real-rooted, for reasons similar to those discussed in Section \ref{sec:roots}. We have verified this conjecture for $n\leq 30$.

\subsection{Central limit theorems} \label{BnCLT}

The main purpose of this subsection is to give two proofs of the following theorem.

\begin{thm} \label{typeBclt} (Limiting distribution, type $B$) The distribution of the coefficients of $B_n^{\pm}(t)$ is asymptotically normal as $n \rightarrow \infty$. For $n \geq 2$, these numbers have mean $n/2$ and for $n \geq 3$, these numbers have variance $(n+1)/12$. \end{thm}

The first proof of Theorem \ref{typeBclt} uses the following proposition.

\begin{prop} \label{useprop} Let $r$ be any positive integer. Then for $n>r$, the rth moment of the type $B_n$ $\pm$-Eulerian distribution equals the rth moment of the type $B_n$ Eulerian distribution.
\end{prop}

\begin{proof} Instead of working with the $r$th moment, we can work with the $r$th falling moment, which can be computed from the relevant generating function by differentiating (with respect to $t$) $r$ times and then setting $t=1$.

From equation (\ref{eq:moments}), we conclude that
\[ B_n^{\pm}(t) = \frac{1}{2} B_n(t) \pm \frac{1}{2} (1-t)^n.\] Now observe that if $n>r$, then differentiating $(1-t)^n$ $r$ times and setting $t=1$ gives $0$.
\end{proof}

We now give our first proof of Theorem \ref{typeBclt}.

\begin{proof} This is immediate from the method of moments, Proposition \ref{useprop}, and the fact (see Theorem 3.4 of \cite{CM} which uses a real-rootedness argument) that the type $B$ Eulerian distribution is asymptotically normal with the claimed mean and variance. \end{proof}

Next we use generating functions to give a second proof of Theorem \ref{typeBclt}.

\begin{proof} The computation is very similar to the first proof of Theorem \ref{thm:clt}. Again, we will establish the theorem only for the coefficients of $ B_{n}^{+}(t) $, since the proof of the case $ B_{n}^{-}(t) $ is analogous. Let $ W_{n}^{+} $ be a random variable taking values in $ \{0,\dots,n\} $ such that
	\begin{equation*}
		\mathbb{P}(W_{n}^{+}=k) = \frac{1}{\left|B_{n}^{+}\right|} \el{B_n}{k}^+.
	\end{equation*}
	Similarly as before, we normalize $ W_{n}^{+} $ by
	\begin{equation*}
		Z_{n}^{+} = \frac{1}{\sqrt{n+1}} \left( W_{n}^{+} - \frac{n}{2} \right).
	\end{equation*}
	Then $ \left|B_{n}^{+}\right| = \frac{1}{2}\left|B_{n}\right| = 2^{n-1}n! $, and so, Theorem \ref{thm:Bmain} tells that the Laplace transform of $ Z_{n}^{+} $ is given by
	\begin{equation*}
		\mathbb{E}[e^{-s Z^{+}_{n}}]
		= \frac{e^{\frac{sn}{2\sqrt{n+1}}} (1-e^{-s/\sqrt{n+1}})^{n+1}}{n!} \sum_{k\geq 0} \left( \left(k+\frac{1}{2}\right)^n + \frac{1}{2^{n}} \right) e^{-ks/\sqrt{n+1}}.
	\end{equation*}
	As before, the prefactor is asymptotically $ e^{\frac{1}{24}s^{2} + \mathcal{O}(n^{-1/2})} \frac{1}{n!} (s/\sqrt{n+1})^{n+1} $ as $ n\to\infty $. Moreover, for each $ t \in (0, 1) $ and $ x \geq 0 $, we have
	\begin{equation*}
		 \int_{k}^{k+1} x^{n} t^x \, dx
		 \leq \left(k+\frac{1}{2}\right)^{n} t^{n}
		 \leq t^{-2} \int_{k+1}^{k+2} x^{n} t^{x} \, dx,
	\end{equation*}
	for all $ k \geq 0 $. This shows that
	\begin{equation*}
		\int_{0}^{\infty} x^{n} e^{-xs/\sqrt{n+1}} \, dx
		\leq \sum_{k\geq 0}\left(k+\frac{1}{2}\right)^n e^{-ks/\sqrt{n+1}}
		\leq e^{2s/\sqrt{n+1}} \int_{0}^{\infty} x^{n} e^{-xs/\sqrt{n+1}} \, dx,
	\end{equation*}
	and so,
	\begin{align*}
		\sum_{k\geq 0}\left(k+\frac{1}{2}\right)^n e^{-ks/\sqrt{n+1}}
		&= e^{\mathcal{O}(n^{-1/2})} \int_{0}^{\infty} x^{n} e^{-xs/\sqrt{n+1}} \, dx \\
		&= e^{\mathcal{O}(n^{-1/2})} n! \left(\frac{\sqrt{n+1}}{s}\right)^{n+1}.
	\end{align*}
	On the other hand,
	\begin{equation*}
		\sum_{k\geq 0} \frac{1}{2^{n}} e^{-ks/\sqrt{n+1}}
		= \frac{1}{1 - \frac{1}{2}e^{-s/\sqrt{n+1}}}
		\leq 2,
	\end{equation*}
	and so, the contribution from this sum is $ \mathcal{O}(1) $. Combining altogether,
	\begin{equation*}
		\mathbb{E}[e^{-s Z^{+}_{n}}]
		= e^{\frac{1}{24}s^{2} + \mathcal{O}(n^{-1/2})} \left[ e^{\mathcal{O}(n^{-1/2})} + \mathcal{O} \left( \frac{(s/\sqrt{n+1})^{n+1}}{n!} \right) \right]
		= e^{\frac{1}{24}s^{2} + \mathcal{O}(n^{-1/2})}.
	\end{equation*}
	Therefore, the convergence $ \mathbb{E}[e^{-s Z^{+}_{n}}] \to e^{\frac{1}{24}s^{2}} $ holds for all $ s > 0 $ and we are done.
\end{proof}

\subsection{Type $B$ riffle shuffling and sign}

There is a notion of type $B$ riffle shuffling, studied by Bergeron and Bergeron \cite{BB}, and also in \cite{F2}. For $a$ odd (the case of interest to us), a type $B$ $a$-riffle shuffle can be defined as follows. Choose integers $j_1,\ldots,j_a$ according to the multinomial distribution \[ P(j_1,\cdots,j_a) = \binom{n}{ j_1, \ldots, j_a} / a^n, \] where $0 \leq j_i \leq n$, $\sum_{i=1}^a j_i=n$.

Given the $j_i$, cut off the top $j_1$ cards, the next $j_2$ cards and so on,
producing $a$ packets (some possibly empty). Turn the even numbered packets face up. Then drop cards one at a time, according to the rule that if there are $A_j$ cards in packet $j$, the next card is dropped from packet $i$ with probability $A_i/(A_1+\cdots+A_a)$. This is done until all cards have been dropped.

In what follows, let $P^B_{n,a}(w)$ denote the probability of a signed permutation $w$ after a type $B$ $a$-shuffle started at the identity. From \cite{BB},
\begin{equation} \label{BBform}
P^B_{n,a}(w) = {n + (a-1)/2 - \des_B(w^{-1}) \choose n} / a^n.
\end{equation}
As with the type A riffle shuffle, an $r$ iterations of an $a$-shuffle gives the same distribution as one iteration of an $a^r$ shuffle.

We give a formula for the probability of a having a type $B$ permutation with sign $1$ after $r$ iterations of $a$-shuffling, starting from the identity permutation. Let $P_{n,a^r}^{B,+}$ denote this probability.

\begin{thm} For any $n$ and $a$ odd, the probability that $r$ iterations of a type B $a$-shuffle yields an element of $B_n$ with sign 1 is:
\[ P_{n,a^r}^{B,+} = \frac{1}{2} + \frac{1}{2 a^{rn}}.\]
\end{thm}

\begin{proof}
Equation (\ref{eq:seriesBpm}) gives
\[
\frac{\sum_{w \in B_n^+} t^{\des_B(w)}}{(1-t)^{n+1}} = \sum_{k \geq 0} \frac{(2k+1)^n+1}{2} t^k.
\]
Taking the coefficient of $t^k$ on both sides gives
\[
\sum_{w \in B_n^+} {n + k - \des_B(w) \choose n} = \frac{(2k+1)^n+1}{2}.
\]
Letting $a=2k+1$, we have
\[
\sum_{w \in B_n^+} {n + (a-1)/2 - \des_B(w) \choose n} = \frac{a^n+1}{2}.
\]
Dividing by $a^n$ and using the fact that the sign of $w$ is equal to the sign of $w^{-1}$ gives
\[
\sum_{w \in B_n^+} {n + (a-1)/2 - \des_B(w^{-1}) \choose n}/a^n = \frac{1}{2}
+ \frac{1}{2 a^n} .
\]
Since $r$ $a$-shuffles is the same as a single $a^r$ shuffle,
the result follows from equation \eqref{BBform}.
\end{proof}

As in the symmetric group case, this means that for large decks of cards, the sign is close to random after a single shuffle.

\subsection{Other one-dimensional characters of $B_n$}

Aside from the trivial and sign character, there are two other one-dimensional characters of $B_n$, and we can easily consider their interaction with descents, following Reiner \cite{Rei}.

Recall from the definition of $\sgn_B$ in \eqref{eq:sgnB} that we can express a signed permutation $w = u^J$ in terms of a permutation $u\in S_n$ such that $|w(i)|=u(i)$ and $J = \{ j : w(j) < 0\}$. We define the characters
\[
 \delta(w) = (-1)^J
\]
and
\[
 \eta(w) = \sgn(u),
\]
and note that from \eqref{eq:sgnB} we have $\sgn_B(w) = \delta(w)\cdot \eta(w)$.
In terms of signed cycles,
\[
\delta(w) = (-1)^{\sum_i m_i(w)},
\]
and
\[
\eta(w) = (-1)^{n - \sum_i n_i(w) - \sum_i m_i(w)},
\]
where we recall that $n_i(w)$ is the number of positive $i$-cycles of $w$ and $m_i(w)$ is the number of negative $i$-cycles of $w$.

Setting $x_i=1, y_i=-1$ in Theorem \ref{vics} gives that
 \[
 \sum_{n \geq 0} \frac{u^n \sum_{w \in B_n} \delta(w) t^{\des_B(w)+1}}{(1-t)^{n+1}} =1+\sum_{k \geq 1} \frac{t^k}{(1-u)} .
 \]
This identity is exactly the same as the one obtained in the proof of Theorem \ref{thm:Bmain} for the joint distribution of $\sgn_B$ with descents, so all our results for $\sgn_B$ carry over to $\delta$ immediately.

To analyze $\eta$, we set $u \mapsto -u, x_i \mapsto -1, y_i \mapsto -1$ in Theorem \ref{vics}, to find
\[
\sum_{n \geq 0} \frac{u^n \sum_{w \in B_n} \eta(w) t^{\des_B(w)+1}}{(1-t)^{n+1}} = 1+ \sum_{k \geq 1} \frac{t^k}{(1-u)} \prod_{m \geq 1} \left(
\frac{1-(-u)^m}{1+(-u)^m} \right)^{N^*(2k-1,2m)}.
\]

From Lemma 1.3.17 of \cite{FNP},
\[
\prod_{m \geq 1} \left(\frac{1-(-u)^m}{1+(-u)^m} \right)^{N^*(2k-1,2m)}
= \frac{1+(2k-1)u}{1+u} .
\]
It follows that
\[
\sum_{n \geq 0} \frac{u^n \sum_{w \in B_n} \eta(w) t^{\des_B(w)+1}}{(1-t)^{n+1}} =1+ \sum_{k \geq 1} t^k \frac{1 + (2k-1)u}{1-u^2}.
\]
Taking the coefficient of $u^n$, $n>0$, it follows that if $n$ is even, then
\[
\sum_{w \in B_n} \eta(w) t^{\des_B(w)} = (1-t)^{n+1} \sum_{k \geq 0} t^k,
\]
exactly as for $\sgn_B$ and for $\delta$.

But if $n$ is odd, something new happens. We get that
\[
\sum_{w \in B_n} \eta(w) t^{\des_B(w)} = (1-t)^{n+1}\sum_{k \geq 0} (2k+1)t^k.
\]
Since
\[
\sum_{w \in B_n} t^{\des_B(w)} = (1-t)^{n+1} \sum_{k \geq 0} (2k+1)^n t^k ,
\]
it follows that
\[
\frac{\sum_{w \in B_n \atop \eta(w)=1} t^{\des_B(w)}}{(1-t)^{n+1}} = \sum_{k \geq 0}  \frac{(2k+1)^n + (2k+1)}{2}t^k.
\]

So all of our main results have analogs for $\eta$, and we leave the details to the interested reader.

We mention that these identities can also be deduced from the following identities, also due to Reiner \cite[Theorems 3.2 and 3.3]{Rei} and proved combinatorially:
\[
 \sum_{w \in B_n} \sgn_B(w) t^{\des_B(w)} = \sum_{w \in B_n} \delta(w) t^{\des_B(w)} = (1-t)^n,
 \]
\[
 \sum_{w \in B_n} \eta(w) t^{\des_B(w)} = \begin{cases} (1-t)^n & \mbox{for $n$ even,} \\
   (1+t)(1-t)^{n-1} & \mbox{for $n$ odd}.
   \end{cases}
\]

\end{document}